\documentclass{amsart}
\usepackage{amsfonts}
\usepackage{bbm}
\usepackage{mathrsfs}
\usepackage{amsfonts}
\usepackage{amsmath}
\usepackage{bigdelim}
\usepackage{multirow}
\usepackage{amssymb}
\usepackage{indentfirst}
\usepackage[dvips]{graphicx}
\usepackage{CJK}
\usepackage{fancyhdr}
\usepackage{amscd,amssymb,amsmath,graphicx,verbatim}
\usepackage[TS1,OT1,T1]{fontenc}

\newtheorem{theorem}{Theorem}[section]
\newtheorem{lemma}[theorem]{Lemma}

\newtheorem{proposition}[theorem]{Proposition}
\newtheorem{problem}[theorem]{Problem}
\theoremstyle{definition}
\newtheorem{definition}[theorem]{Definition}
\newtheorem{example}[theorem]{Example}

\theoremstyle{remark}
\newtheorem{remark}[theorem]{Remark}
\numberwithin{equation}{section}

\begin{document}

\title{Entropy of a semigroup of maps from a set-valued view}

\author{Bingzhe Hou}
\address{Bingzhe Hou, College of Mathematics , Jilin university, 130012, Changchun, P. R. China} \email{houbz@jlu.edu.cn}
\author{Xu Wang}
\address{Xu Wang, College of Mathematics , Jilin university, 130012, Changchun, P. R. China} \email{chengjiuyouxiu@sina.com}

\date{}
\subjclass[2000]{37A35, 37B40, 54H20, 37C85.}
\keywords{topological entropy, semigroups, entropy-like invariants, Hausdorff metric.}
\thanks{}
\begin{abstract}
In this paper, we introduce a new entropy-like invariant, named Hausdorff metric entropy, for finitely generated semigroups acting on compact metric spaces from a set-valued view and study its properties.
We establish the relation between Hausdorff metric entropy and topological entropy of a semigroup defined by Bi\'{s}. Some examples with positive or zero  Hausdorff metric entropy are given. Moreover,
some notions of chaos are also well generalized for finitely generated semigroups from a set-valued view.
\end{abstract}
\maketitle

\section{Introduction}

Suppose that$(X,d)$ is a compact metric space. Let $F=\{f_1,f_2,\ldots,f_p\}$ be a $p$-tuple of continuous maps from $X$ to itself, and $G$ be the semigroup generated by $F$. In this paper, we are interested in the dynamical system $(X,F)$(or $(X,G,F)$). In classical discrete
topological dynamics, the concept of topological entropy for a continuous transformation plays an important role. This notion was introduced by Adler, Konheim and McAnderew in \cite{Adler} as an invariant of topological conjugacy. Later, Bowen \cite{Bowen} and
Dinaburg \cite{Dina} gave an equivalent description to this notion on metric space, namely (Bowen's) metric entropy. For a continuous map $f$ on a compact metric space, denote $h(f)$ the topological entropy (or equivalently, metric entropy) of $f$. For finitely generated semigroups acting on compact metric spaces, there had been some entropy-like invariants introduced. For instance, Friedland introduced the entropy of a graph in \cite{Fri}, which is called Friedland entropy now; Bi\'{s} in \cite{Bis} introduced topological entropy of a semigroup, inverse image entropy, preimage relation entropy and point entropy, and established their relations.

The aim of the present paper is to introduce a new entropy-like invariant, named Hausdorff metric entropy, for finitely generated semigroups acting on compact metric spaces from a set-valued view. We will give the definition in next section and study the basic properties of Hausdorff metric entropy in section 3. In section 4, we will establish the relations between Hausdorff metric entropy, topological entropy of a semigroup defined by Bi\'{s} and $\max \{h(f_i); f_i\in F\}$. In fact, two examples are given. One shows that the Hausdorff metric entropy of a tuple $F$ can be positive when each element in $F$ has zero entropy; the other one shows that the Hausdorff metric entropy of a tuple $F$ can be zero when there exists an element in $F$ with positive entropy. At the last section, some remarks and problems are given. In particular, some notions of chaos are also well generalized for finitely generated semigroups.

\section{Definition of Hausdorff metric entropy}
Suppose that$(X,d)$ is a compact metric space. Let $F=\{f_1,f_2,\ldots,f_p\}$ be a $p$-tuple of continuous maps from $X$ to itself, and $G$ be the semigroup generated by $F$.
Then for every $n\geq1$,
$$
F^n=\{g_n\circ g_{n-1}\circ\cdots\circ g_1; g_j\in\{f_1,f_2,\ldots,f_p\}, \ for \ all \ j=1,2,\ldots,n. \}
$$
is a finite set of continuous maps at most $p^n$. Put $F^0=\{id_X\}$, where $id_X$ is the identity map on $X$.
Notice that for any point $x\in X$, $F(x)=\{f_1(x), f_2(x)\ldots,f_n(x)\}$ is a non-empty compact subset of $X$,
we can give a new definition of metric entropy from a set-valued view. First of all, let us review some notions with respect to set-valued spaces.

Let
$$
\mathcal{K}(X)=\{K; K \ is \ a \ non-empty \ compact \ subset \ of \ X \}.
$$
Then the metric $d$ on $X$ induces a metric on $\mathcal{K}(X)$. For any $A,B\in\mathcal{K}(X)$, define
\begin{equation*}
dist(A,B)
=\sup\limits_{x\in A}d(x,B)
=\sup\limits_{x\in A}{ \inf\limits_{y\in B}d(x,y)}.
\end{equation*}
Furthermore, we define the Hausdorff metric $d_H$ by
\begin{equation*}
d_H(A,B)=\max\{dist(A,B), dist(B,A)\} .
\end{equation*}
Moreover, $(\mathcal{K}(X),d_H)$ is a compact metric space. In addition, the topology induced by the Hausdorff metric $d_H$ on $\mathcal{K}(X)$
coincides with the Vietoris topology \cite{KT}. Denote $\mathfrak{J}$ the topology induced by the metric $d$ on $X$. The Vietoris topology $\mathfrak{J}_V$ is generated by the
base $\mathfrak{B}$ consisting of sets of the form
$$
\mathcal{B}(U_{1},U_{2},\ldots,U_{m})=\{K\in\mathcal{K}(X); K\subseteq\bigcup\limits_{i=1}^{m}{U_{i}}, \
and \ K\bigcap U_{i}\neq
\emptyset, \ 1\leq i\leq m.\},
$$
where $U_{1},U_{2},\ldots,U_{m}$ are non-empty open subsets of $X$.

Given any $n\geq0$. For any $x\in X$,
$$
F^n(x)=\{g_n\circ\cdots\circ g_1(x); g_j\in\{f_1,f_2,\ldots,f_p\}, \ for \ all \ j=1,2,\ldots,n.\}\in \mathcal{K}(X).
$$
Then $F^n$ can be seemed as a continuous map from $X$ to $\mathcal{K}(X)$.
In fact, for any subset $A$ in $X$,
$$
F^n(A)=\bigcup\limits_{a\in A}F^n(a),
$$
In particular, if $A\in\mathcal{K}(X)$ we have $F^n(A)\in\mathcal{K}(X)$. Thus, $F$ induced naturally a map $\widetilde{F}: \mathcal{K}(X)\rightarrow\mathcal{K}(X)$ defined by
$$
\widetilde{F}^n(A)=F^n(A), \ \ \ for \ any \ A\in\mathcal{K}(X).
$$
Moreover, $\widetilde{F}^n$ is a continuous map from $(\mathcal{K}(X),d_H)$ to itself.
In addition, there exists a natural isometric embedding $\varphi:X\rightarrow\mathcal{K}(X)$ defined by
$$
\varphi(x)=\{x\}, \ \ for \ every \ x\in X,
$$
since for any $x,y\in X$,
$$
d(x,y)=d_H(\{x\},\{y\}).
$$
Then $F^n=\widetilde{F}^n\circ\varphi$. The following lemma guaranteed the continuity of $\widetilde{F}^n$~($F^n$).
\begin{lemma}
Let $(X,d)$ be a compact metric space and $F=\{f_1,f_2,\ldots,f_p\}$ be a $p$-tuple of continuous maps from $X$ to itself.
Then $\widetilde{F}:\mathcal{K}(X)\rightarrow\mathcal{K}(X)$ is a continuous map.
\end{lemma}
\begin{proof}
Given arbitrary element $\mathcal{B}(U_{1},U_{2},\ldots,U_{m})$ in the base $\mathfrak{B}$. It follows from the continuity of $f_1,f_2,\ldots,f_p$ that
for each $1\leq i\leq m$ and each $1\leq j\leq p$, $f_j^{-1}(U_i)$ is a non-empty open subset of $X$ and consequently
$$
F^{-1}(U_i)=\bigcup\limits_{j=1}^pf_j^{-1}(U_i)
$$
is a non-empty open subset of $X$. Thus,
$$
\widetilde{F}^{-1}\mathcal{B}(U_{1},U_{2},\ldots,U_{m})=\mathcal{B}(F^{-1}(U_{1}),F^{-1}(U_{2}),\ldots,F^{-1}(U_{m}))\in \mathfrak{B},
$$
which implies $\widetilde{F}$ is a continuous map.
\end{proof}

Now let us introduce a metric entropy for a finite set of continuous maps from a set-valued view likes the concept of Bowen's metric entropy.
For any $x,y\in X$, denote
$$
d_H^n(x,y)=\max\limits_{0\leq i\leq n}d_H(F^i(x), F^i(y))
$$
For any $n\geq 0$ and any $\epsilon>0$, a subset $M$ in
$X$ is said to be a Hausdorff metric $(n,\epsilon)$-spanning set of $X$ with respect to $F$, if for each $x\in X$, there exists a point $y\in M$ such that
$$
d_H^n(x,y)<\epsilon.
$$
It follows from the compactness of $X$ and the continuity of $F$ that,
$$
r_H(n, \epsilon, X, F)=\min\{ Card(M);M \ \text{is  a  Hausdorff  metric} \ (n,\epsilon)-\text{spanning  set of} \ X\}
$$
is a finite positive integer.
\begin{definition}
Let
$$
h_{H}(G, F)=\lim\limits_{\epsilon\rightarrow0^+}\limsup\limits_{n\rightarrow+\infty}\frac{1}{n}\log{r_H(n, \epsilon, X, F)}.
$$
The quantity $h_{H}(G, F)$ is called the Hausdorff metric entropy of a semigroup $G$ generated by $F$.
\end{definition}
We can also describe the Hausdorff metric entropy of a semigroup $G$ generated by $F$ in terms of Hausdorff metric $(n,\epsilon)$-separated sets. Namely, for any $n\geq 0$ and any $\epsilon>0$, a subset $E$ in
$X$ is said to be a Hausdorff metric $(n,\epsilon)$-separated set of $X$ with respect to $F$, if for any distinct $x,y\in E$,
$$
d_H^n(x,y)\geq\epsilon.
$$
It follows from the compactness of $X$ and the continuity of $F$ that,
$$
s_H(n, \epsilon, X, F)=\max\{Card(E); E \ \text{is  a Hausdorff  metric} \ (n,\epsilon)-\text{separated  set of} \ X\}
$$
is a finite positive integer.
\begin{proposition}
For any semigroup $G$ generated by a finite set $F$, the following equality
$$
\lim\limits_{\epsilon\rightarrow0^+}\limsup\limits_{n\rightarrow+\infty}\frac{1}{n}\log{r_H(n, \epsilon, X, F)}=\lim\limits_{\epsilon\rightarrow0^+}\limsup\limits_{n\rightarrow+\infty}\frac{1}{n}\log{s_H(n, \epsilon, X, F)}
$$
holds.
\end{proposition}
\begin{proof}
Let $E$ be a Hausdorff metric $(n,\epsilon)$-separated set of maximal cardinality. Then $E$ is also a Hausdorff metric $(n,\epsilon)$-spanning set of $X$, and consequently
$$
r_H(n, \epsilon, X, F)\leq s_H(n, \epsilon, X, F).
$$

Now let $M$ be a Hausdorff metric $(n,\epsilon/2)$-spanning set of minimal cardinality. Define a map $\varphi: E\rightarrow M$ by choosing for each $x\in E$, some point
$\varphi(x)\in M$ such that
$$
d^n_H(x, \varphi(x))<\frac{\epsilon}{2}.
$$
Then $\varphi$ is injective and hence the cardinality of $E$ is not greater than that of $M$, i.e.,
$$
s_H(n, \epsilon, X, F)\leq r_H(n, \frac{\epsilon}{2}, X, F).
$$
Therefor, by
$$
r_H(n, \epsilon, X, F)\leq s_H(n, \epsilon, X, F)\leq r_H(n, \frac{\epsilon}{2}, X, F),
$$
we have
$$
\lim\limits_{\epsilon\rightarrow0^+}\limsup\limits_{n\rightarrow+\infty}\frac{1}{n}\log{r_H(n, \epsilon, X, F)}=\lim\limits_{\epsilon\rightarrow0^+}\limsup\limits_{n\rightarrow+\infty}\frac{1}{n}\log{s_H(n, \epsilon, X, F)}.
$$
\end{proof}

\section{Fundamental properties of Hausdorff metric entropy}

First of all, one can see when a finite set of continuous transformation contains only one element, the  Hausdorff metric entropy is actually the classical topological entropy.

\begin{proposition}\label{deng}
Let $(X,d)$ be a compact metric space and $G$ be a semigroup generated by $F=\{f\}$. Then $h_H(G, F)=h(f)$.
\end{proposition}
\begin{proof}
For any $x\in X$ and any $n>0$, we have $F^n(x)=\{f^n(x)\}$. Then for any $x,y\in X$,
$$
d(F^n(x),F^n(y))=d_H(\{f^n(x)\},\{f^n(y)\})=d(f^n(x), f^n(y)).
$$
Thus, $h_H(G, F)=h(f)$.
\end{proof}

Next, similar to the research of the topological entropy of a single map, we will study the basic properties of the Hausdorff metric entropy for finitely generated semigroups.

\begin{proposition}
Let $(X,d)$ be a compact metric space and $G$ be a semigroup generated by a finite set $F=\{f_1,f_2,\ldots,f_p\}$. Then $h_H(G, F)\geq0$.
\end{proposition}
\begin{proof}
It follows from $r_H(n,\epsilon,X,F)>0$ that
$$
h_{H}(G, F)=\lim\limits_{\epsilon\rightarrow0^+}\limsup\limits_{n\rightarrow+\infty}\frac{1}{n}\log{r_H(n, \epsilon, X, F)}\geq0.
$$
\end{proof}

\begin{proposition}
For every positive integer $m$, $h_H(G, F^m)=m\cdot h_H(G, F).$
\end{proposition}
\begin{proof}
Firstly, one can see that $G$ is also the semigroup generated by $F^m$.
Given any $m>0$, $n>0$ and $\epsilon>0$. A Hausdorff metric $(mn,\epsilon)$-spanning set of $X$ with respect to $F$  must be a Hausdorff metric $(n,\epsilon)$-spanning set of $X$ with respect to $F^m$, then
$$
r_H(n,\epsilon,X,F^m)\leq r_{H}(mn,\epsilon,X,F).
$$
Consequently,
$$
\frac{1}{n}\log r_H(n,\epsilon,X,F^m)\leq \frac{m}{mn}\log r_{H}(mn,\epsilon,X,F).
$$
Thus,
$$
h_H(G, F^m)\leq m\cdot h_H(G, F).
$$

On the other hand, it follows from the uniform continuity of $F$(~$\widetilde{F}$~) that for any $\epsilon>0$, there exists $\delta>0$ such that if $d(x,y)<\delta$, $d_H^m(x,y)<\epsilon$.
Then a Hausdorff metric $(n,\delta)$-spanning set of $X$ with respect to $F^m$ must be a Hausdorff metric $(mn,\epsilon)$-spanning set of $X$ with respect to $F$,  and consequently
$$
r_H(n,\delta,X,F^m)\geq r_{H}(mn,\epsilon,X,F).
$$
Thus,
$$
h_H(G, F^m)\geq m\cdot h_H(G, F).
$$
\end{proof}

\begin{definition}
Let $(X,d)$ be a compact metric space and $G$ be a semigroup generated by a finite set $F=\{f_1,f_2,\ldots,f_p\}$.
$Y\subseteq X$ is called an invariant subset of $X$, if for any $y\in Y$, $F(y)\subseteq Y$, i.e., $F(Y)\subseteq Y$.
Moreover, if $Y\subseteq X$ is an invariant closed subset of $X$, we call $(Y,F|_Y)$ is a subsystem of $(X,F)$.
\end{definition}

\begin{proposition}
Let $(X,d)$ be a compact metric space and $G$ be a semigroup generated by a finite set $F=\{f_1,f_2,\ldots,f_p\}$.
If $Y\subseteq X$ is an invariant closed subset of $X$, then $h_H(G|_Y, F|_Y)\leq h_H(G, F)$.
\end{proposition}
\begin{proof}
Given any $n>0$ and $\epsilon>0$. A Hausdorff metric $(n,\epsilon)$-separated set of $Y$ with respect to $F|_Y$  must be a Hausdorff metric $(n,\epsilon)$-separated set of $X$ with respect to $F$, then
$$
s_H(n,\epsilon,Y,F|_Y)\leq s_{H}(n,\epsilon,X,F).
$$
Consequently,
$$
\frac{1}{n}\log s_H(n,\epsilon,Y,F|_Y)\leq \frac{1}{n}\log s_{H}(n,\epsilon,X,F).
$$
Thus,
$$
h_H(G|_Y, F|_Y)\leq h_H(G, F).
$$
\end{proof}

\begin{definition}
Let $(X,d)$ and $(Y,\rho)$ be two compact metric spaces. Let $G$ and $G'$ be two semigroups generated by finite set $F=\{f_1,f_2,\ldots,f_p\}$ and finite set $F'=\{f'_1,f'_2,\ldots,f'_q\}$ respectively.
If there exists a continuous surjective $T:X\rightarrow Y$ such that for any $x\in X$,
$$
\{Tf_1(x),Tf_{2}(x)\ldots,Tf_p(x)\}=\{f'_1T(x),f'_{2}T(x)\ldots,f'_pT(x)\}, \  i.e. \ TF(x)=F'(Tx),
$$
then $T$ is called a topological semiconjugacy from $(X,F)$ to $(Y,F')$. Moreover, if $T$ is a homeomorphism, we call $T$ a topological conjugacy from $(X,F)$ to $(Y,F')$.
\end{definition}

\begin{theorem}
Let $(X,d)$ and $(Y,\rho)$ be two compact metric spaces. Let $G$ and $G'$ be two semigroups generated by finite set $F=\{f_1,f_2,\ldots,f_p\}$ and finite set $F'=\{f'_1,f'_2,\ldots,f'_q\}$ respectively.
If there exists a topological semiconjugacy $T$ from $(X,F)$ to $(Y,F')$,  then $h_H(G, F)\geq h_H(G', F')$.
Moreover, $T$ is a topological conjugacy, then $h_H(G, F)=h_H(G', F')$.
\end{theorem}
\begin{proof}
For any $n>0$, it follows from the uniform continuity of $T$ that for any $\epsilon>0$,
there exsits $\delta>0$ such that if $d_H^n(x,y)<\delta$, $d_H^n(Tx,Ty)<\epsilon$.
Then, if $M$ is a Hausdorff metric $(n,\delta)$-spanning set of $X$ with respect to $F$, $T(M)$ must be a Hausdorff metric $(n,\epsilon)$-spanning set of $Y$ with respect to $F'$,  and consequently
$$
r_H(n,\delta,X,F)\geq r_{H}(n,\epsilon,Y,F').
$$
Thus,
$$
h_H(G, F)\geq h_H(G', F').
$$

Notice that if $T$ is a topological conjugacy from $(X,F)$ to $(Y,F')$, $T^{-1}$ is a topological conjugacy from $(Y,F')$ to $(X,F)$. Therefor,
$$h_H(G, F)= h_H(G', F')$$
\end{proof}

Let $(X,d)$ and $(Y,\rho)$ be two compact metric spaces. Let $G$ and $G'$ be two semigroups generated by finite sets $F=\{f_1,f_2,\ldots,f_p\}$ and $F'=\{f'_1,f'_2,\ldots,f'_q\}$ respectively.
Now consider the Cartesian product space $X\times Y$ with metric $\gamma$ defined by, for any $(x,y), (u,v)\in X\times Y$,
$$
\gamma((x,y), (u,v))=\max\{d(x,u), \rho(y,v)\}.
$$
Denote
$$
F\times F'=\{g\times g'; g\in F, g\in F' \}.
$$
Then $F\times F'$ is a finite set of continuous maps on $X\times Y$, and it is easy to see for any $n>0$, $(F\times G)^n=F^n\times G^n$. Moreover, denote $G\times G'$ the semigroup generated by $F\times F'$ and
$h_H(G\times G',F\times F')$ the Hausdorff metric entropy of Cartesian product system $(X\times Y,F\times F')$.

\begin{lemma}
For any $A\times B, C\times D\in \mathcal{K}(X)\times \mathcal{K}(Y)\subseteq\mathcal{K}(X\times Y)$, we have
$$
\gamma_H(A\times B, C\times D)=\max\{d_H(A,C), \rho_H(B,D)\}.
$$
\end{lemma}
\begin{proof}
Firstly, we will prove for any $(x,y)\in X\times Y$ and any $E\times K\in \mathcal{K}(X)\times \mathcal{K}(Y)$,
$$
\gamma((x,y), E\times K)=\max\{d(x,E), \rho(y,K)\}.
$$
One hand, since $E\times K$ is a compact subset of $X\times Y$, there exists a point $(e,k)\in E\times K$ such that
$$
\gamma((x,y), E\times K)=\gamma((x,y), (e,k))=\max\{d(x,e), \rho(y,k)\}\geq \max\{d(x,E), \rho(y,K)\}.
$$
On the other hand , since $E$ and $K$ are compact subsets in $X$ and $Y$ respectively, there exist $e'\in E$ and $k'\in K$ such that
$$
d(x,E)=d(x,e'), \ \ \  \ \rho(y,K)=\rho(y,k').
$$
Then
\begin{align*}
\max\{d(x,E), \rho(y,K)\}&=\max\{d(x,e'), \rho(y,k')\} \\
&=\gamma((x,y), (e',k')) \\
&\geq\gamma((x,y), E\times K).
\end{align*}
Therefor,
$$
\gamma((x,y), (E,K))=\max\{d(x,E), \rho(y,K)\}
$$
and consequently,
\begin{align*}
&\gamma_H(A\times B, C\times D)  \\
=&\max\{\sup\limits_{(a,b)\in A\times B}\gamma((a,b),C\times D), \sup\limits_{(c,d)\in C\times D}\gamma((c,d),A\times B)\} \\
=&\max\{\sup\limits_{(a,b)\in A\times B}\max\{d(a,C), \rho(b,D)\}, \sup\limits_{(c,d)\in C\times D}\max\{d(c,A), \rho(d,B)\}\} \\
=&\max\{dist_d(A,C), dist_{\rho}(B,D), dist_{d}(C,A),dist_{\rho}(D,B)\} \\
=&\max\{d_H(A,C), \rho_H(B,D)\}.
\end{align*}
\end{proof}

\begin{theorem}
The equality
$$
h_H(G\times G',F\times F')=h_H(G, F)+h_H(G', F')
$$
holds.
\end{theorem}
\begin{proof}
By above lemma, for any $(a,b),~(c,d)\in X\times Y$ and any $n>0$,
$$
\gamma_H((F\times G)^n(a,b),(F\times G)^n(c,d))=\max\{d_H(F^n(a),F^n(c)), \rho_H(G^n(b),G^n(d))\}.
$$
Then
$$
\gamma_H^n((a,b),~(c,d))=\max\{d_H^n(a,c), \rho_H^n(b,d)\}.
$$

One hand, if $M$ is a Hausdorff metric $(n,\epsilon)$-spanning set of $X$ with respect to $F$ and $M'$ is a Hausdorff metric $(n,\epsilon)$-spanning set of $Y$ with respect to $F'$,
then $M\times M'$ is a Hausdorff metric $(n,\epsilon)$-spanning set of $X\times Y$ with respect to $F\times F'$, and consequently
$$
r_{H}(n,\epsilon,X\times Y,F\times F')\leq r_{H}(n,\epsilon,X,F)\cdot r_{H}(n,\epsilon,Y, F'),
$$
Therefor,
$$
h_H(G\times G',F\times F')\leq h_H(G, F)+h_H(G',F').
$$

On the other hand, if $E$ is a Hausdorff metric $(n,\epsilon)$-separated set of $X$ with respect to $F$ and $E'$ is a Hausdorff metric $(n,\epsilon)$-separated set of $Y$ with respect to $F'$,
then $E\times E'$ is a Hausdorff metric $(n,\epsilon)$-separated set of $X\times Y$ with respect to $F\times F'$, and consequently
$$
s_{H}(n,\epsilon,X\times Y,F\times F')\geq s_{H}(n,\epsilon,X,F)\cdot s_{H}(n,\epsilon,Y, F'),
$$
Therefor,
$$
h_H(G\times G',F\times F')\geq h_H(G, F)+h_H(G',F').
$$

Thus,
$$
h_H(G\times G',F\times F')=h_H(G, F)+h_H(G',F').
$$
\end{proof}

\section{Further discussions}

We will give an example to show that the Hausdorff metric entropy of a tuple $F$ can be positive when each element in $F$ has zero entropy. In fact, the following example is contained in a result of
Bi\'{s} and Urba\'{n}ski (Theorem 4.1 in \cite{BU}, where they discussed Bi\'{s}'s topological entropy of a semigroup).

\begin{example}\label{da}
Denote $X$ by the unit interval $[0,1]$. Let $F=\{f_1,f_2\}$, where $f_1, f_2: X\rightarrow X$ defined by
$$
f_1(x)=\left\{\begin{array}{cc}
x, \ \ \ &\mbox{if \ $0\leq x\leq 1/3$} \\
3x-2/3, \ \ \ &\mbox{if \ $1/3\leq x\leq 4/9$} \\
3x/5+2/5, \ \ \ &\mbox{if \ $4/9\leq x\leq 1$}
\end{array}\right.
$$
and
$$
f_2(x)=\left\{\begin{array}{cc}
3x/5, \ \ \ &\mbox{if \ $0\leq x\leq 5/9$} \\
3x-4/3, \ \ \ &\mbox{if \ $5/9\leq x\leq 2/3$} \\
x, \ \ \ &\mbox{if \ $2/3\leq x\leq 1$}
\end{array}\right.
$$
Then
$$
h(f_1)=h(f_2)=0 \  \ and \ \ h_H(G,F)\geq \log2>0.
$$
\end{example}

Denote $Y$ the subinterval $[1/3,2/3]$. Then $f_1^{-1}(Y)=[1/3,4/9]\subseteq Y$,  $f_2^{-1}(Y)=[5/9,2/3]\subseteq Y$ and $d(f_1^{-1}(Y), f_2^{-1}(Y))=1/9$.
Fix $\epsilon\in(0,1/15)$. Since every map $g: X\rightarrow X$, $g\in G$, is a homeomorphism, one can select for every $g\in G$
exactly one point $z_g\in g^{-1}(Y)$. For every $n\geq0$, consider the set $A_n=\{z_g; g\in F^n\}$. We will show that $A_n$ is a Hausdorff metric
$(n,\epsilon)$-separated set consisting of exactly $2^n$ elements.

Now it suffices to prove that, for two arbitrary elements $g\neq h$ in $F^n$, $d_H^n(z_g,z_h)\geq1/15$. Write $g=g_n\circ g_{n-1}\circ\cdots\circ g_1$ and
$h=h_n\circ h_{n-1}\circ\cdots\circ h_1$, where $g_j, h_j\in \{f_1,f_2\}$ for all $j=1,2,\ldots,n$. Since $g\neq h$, there exists $k\in\{1,2,\ldots,n\}$ such that
$g_1=h_1$, $g_2=h_2$, $\ldots$, $g_{k-1}=h_{k-1}$ and $g_k\neq h_k$. Then
$$
g_{k-1}\circ\cdots\circ g_1(z_g)\in g_k^{-1}(Y)
$$
and
$$
g_{k-1}\circ\cdots\circ g_1(z_h)\in h_k^{-1}(Y).
$$
Hence,
$$
d(g_{k-1}\circ\cdots\circ g_1(z_g), g_{k-1}\circ\cdots\circ g_1(z_h)\geq\frac{1}{9}>\frac{1}{15}>\epsilon.
$$
Given any $w=w_n\circ w_{n-1}\circ\cdots\circ w_1\in F^n$, where $w_j\in \{f_1,f_2\}$ for all $j=1,2,\ldots,n$. There are two cases as follows.

Case 1. $g_1=w_1$, $g_2=w_2$, $\ldots$, $g_{k-1}=w_{k-1}$. Then
$$
w_{k-1}\circ\cdots\circ w_1(z_g)=g_{k-1}\circ\cdots\circ g_1(z_g)
$$
and
$$
w_{k-1}\circ\cdots\circ w_1(z_h)=h_{k-1}\circ\cdots\circ h_1(z_h).
$$

Case 2. there exists $t\in\{1,2,\ldots,k-1\}$ such that
$g_1=w_1$, $g_2=w_2$, $\ldots$, $g_{t-1}=w_{t-1}$ and $g_t\neq w_t$. Then
$$
g_{t-1}\circ\cdots\circ g_1(z_g)\in g_t^{-1}(Y) \ and \ g_{t-1}\circ\cdots\circ g_1(z_h)\in g_t^{-1}(Y),
$$
and consequently
$$
w_t\circ g_{t-1}\circ\cdots\circ g_1(z_g), \ w_t\circ g_{t-1}\circ\cdots\circ g_1(z_h)\in [\frac{11}{15}, \frac{4}{5}]\bigcup [\frac{1}{5}, \frac{4}{15}].
$$
Furthermore,
$$
w_{k-1}\circ\cdots\circ w_t\circ w_{t-1}\circ\cdots\circ w_1(z_g), w_{k-1}\circ\cdots\circ w_t\circ w_{t-1}\circ\cdots\circ w_1(z_h)\in [\frac{11}{15}, 1]\bigcup [0, \frac{4}{15}].
$$
This implies
$$
d(g_{k-1}\circ\cdots\circ g_1(z_g), w_{k-1}\circ\cdots\circ w_1(z_h)\geq d(Y, [\frac{11}{15}, 1]\bigcup [0, \frac{4}{15}])=\frac{1}{15}>\epsilon.
$$
and
$$
d(g_{k-1}\circ\cdots\circ g_1(z_h), w_{k-1}\circ\cdots\circ w_1(z_g)\geq d(Y, [\frac{11}{15}, 1]\bigcup [0, \frac{4}{15}])=\frac{1}{15}>\epsilon.
$$

Therefor,
$$
d_{H}(F^{k-1}(z_g), F^{k-1}(z_h))>\epsilon,
$$
which implies that $z_g$ and $z_h$ are Hausdorff metric $(n,\epsilon)$-separated. Notice that the map $g\rightarrow z_g$ is bijective, then
$$
s_H(n, \epsilon, X, F)\geq Card(F^n)=2^n.
$$
In consequence,
$$
h_H(G,F)\geq \log2>0.
$$

Now we will discuss the relation between Hausdorff metric entropy and Bi\'{s}'s topological entropy of a semigroup.  Let us review the concept of the topological entropy $h(G,F)$ defined by Bi\'{s} in \cite{Bis}. Let $(X,d)$ be a compact metric space and $G$ be a semigroup generated by a finite set $F=\{f_1,f_2,\ldots,f_p\}$. In \cite{Bis}, Bi\'{s} assumed the identity map is in $F$. We will show an equivalent version of this definition without the assumption.
Following \cite{Ghys} we will say that two points $x,y\in X$ are $(n, \epsilon)$-separated by G (with respect to the metric $d_{max}^n$) if there exists $g\in F^n$ such that $d(g(x), g(y))\geq\epsilon$, e.g.
$$
d_{max}^n(x,y)=\max\{d^k(x,y); 0\leq k\leq n\}\geq\epsilon,
$$
where
$$
d^k(x,y)=\max\{d(g(x),g(y)); g\in F^k\}.
$$

A subset $A$ of $X$ is $(n, \epsilon)$-separated if any two distinct points of $A$ have this property. Write
$$
s(n, \epsilon, X)=\max\{Card(A); A \ is \ an \ (n, \epsilon)-separated \ subest \ of \ X\},
$$
and consequently, define
$$
h(G,F)=\lim\limits_{\epsilon\rightarrow0^+}\limsup\limits_{n\rightarrow+\infty}\frac{1}{n}\log s(n, \epsilon, X).
$$

The topological entropy $h(G,F)$ also can be described by $(n, \epsilon)$-spanning subset. A subset $B$ of $X$ is called  $(n, \epsilon)$-spanning if for any $x\in X$, there exists a point $a\in A$ such that
$$
d_{max}^n(x,a)<\epsilon.
$$
Write
$$
r(n, \epsilon, X)=\min\{Card(A); A \ is \ an \ (n, \epsilon)-spanning \ subest \ of \ X\}.
$$
One can see
$$
h(G,F)=\lim\limits_{\epsilon\rightarrow0^+}\limsup\limits_{n\rightarrow+\infty}\frac{1}{n}\log r(n, \epsilon, X).
$$

Notice that for any $x,y\in X$ and $k\in \mathbb{N}$,
$$
d_H(F^k(x),F^k(y))\leq d^k(x,y)
$$
and consequently
$$
d_H^k(x,y)\leq d_{max}^k(x,y).
$$
Then an $(n, \epsilon)$-spanning subset must be a Hausdorff metric $(n, \epsilon)$-spanning subset. Thus, we have the following result.
\begin{proposition}\label{HB}
Let $(X,d)$ be a compact metric space and $G$ be a semigroup generated by a finite set $F=\{f_1,f_2,\ldots,f_p\}$. Then
$$
h_H(G,F)\leq h(G,F).
$$
\end{proposition}

The following example will tell us that $h_H(G,F)$ could be strictly less than $h(G,F)$. In fact, it shows that the Hausdorff metric entropy of a tuple $F$ can be zero when there exists an element in $F$ with positive entropy.

\begin{example}\label{xiao}
Let $f$ be a minimal homeomorphism with positive topological entropy on a compact metric space $X$. Let $F=\{f, id_X\}$
and $G$ be the semigroup generated by $F$. Then $h_H(G,F)=0$.
\end{example}

The theorem of Jewett-Krieger \cite{Jew, Kri} assures the existence of minimal systems with positive topological entropy, more concrete examples can be found in \cite{Hahn, Gri, Rees, Beg}.
Now it suffices to prove the following result.

\begin{proposition}
Let $f$ be a minimal homeomorphism on a compact metric space $X$. Let $F=\{f, id_X\}$
and $G$ be the semigroup generated by $F$. Then $h_H(G,F)=0$.
\end{proposition}
\begin{proof}
Notice that for any $x\in X$ and any $k\in\mathbb{N}$,
$$
F^{k}(x)=\{x, f(x),\ldots, f^{k}(x)\}
$$
and for any $n\geq k$,
$$
F^n(x)\supseteq F^k(x).
$$

Given any $\epsilon>0$. For each $x\in X$, by the minimality of $f$, there exists a positive integer $N_x$ such that $F^{N_x}(x)$ is an $\epsilon/2$-net, i.e.,
$$
d_H(F^{N_x}(x), X)<\frac{\epsilon}{2}.
$$
According to the compactness of $X$ and the continuity of $\widetilde{F}$, there  exists a positive integer $N$ such that
$$
d_H(F^{N}(x), X)<\frac{\epsilon}{2}, \ \ \ for \ any \ x\in X.
$$
Furthermore, for any $x,y\in X$ and $n\geq N$,
$$
d_H(F^{n}(x), F^{n}(y))<\epsilon.
$$
Now one can see that a Hausdorff metric $(N,\epsilon)$-spanning set is also a Hausdorff metric $(n,\epsilon)$-spanning set for every $n\geq N$. Thus
$$
h_H(G,F)=\lim\limits_{\epsilon\rightarrow 0^+}\limsup\limits_{n\rightarrow +\infty}\frac{1}{n}r_H(n, \epsilon, X, F)=0.
$$
\end{proof}

\begin{remark}\label{rel}
Following from \cite{Bis}, for any semigroup $G$ generated by a finite set $F$,
$$
h(G,F)\geq \max \{h(f_i); f_i\in F\}.
$$
One may ask "What is the relation between the Hausdorff metric entropy of $F$ and $\max \{h(f_i); f_i\in F\}$?" Proposition \ref{deng}, Example \ref{da} and Example \ref{xiao} show that $h_H(G,F)$ are possible to be less than, equal to, or more than $\max \{h(f_i); f_i\in F\}$.
\end{remark}
\section{Some remarks and problems}

From a set-valued view, we can also well generalize Li-Yorke chaos (\cite{Li}) and distributional chaos (There are three versions of distributional chaos denoted by DC1,
DC2 and DC3 in brief. DC1 was originally introduced in \cite{Sch},
and the generalizations DC2 and DC3 were introduced in
\cite{BSS, S-S2004}.).

\begin{definition}
Let $(X,d)$ be a compact metric space and $G$ be a semigroup generated by a finite set $F=\{f_1,f_2,\ldots,f_p\}$.
$\{x,y\}\subset X$ is called a Hausdorff metric Li-Yorke pair, if
\begin{equation*}
\limsup\limits_{n\rightarrow+\infty}d_H(F^{n}(x),F^{n}(y))>0 \ \ and \ \
\liminf\limits_{n\rightarrow+\infty}d_H(F^{n}(x),F^{n}(y))=0.
\end{equation*}
Moreover, if there exists an uncountable subset $S$ of $X$ such that for any distinct points $x,y\in S$, $\{x,y\}$ is a Hausdorff metric Li-Yorke pair,
the system $(X,F)$ is called Hausdorff metric Li-Yorke chaotic and the set $S$ is called a Hausdorff metric scrambled set of $X$ with respect to $F$.
\end{definition}

Let $(X,d)$ be a compact metric space and $G$ be a semigroup generated by a finite set $F=\{f_1,f_2,\ldots,f_p\}$.
For any $\{x,y\}\subset X$ and any $n\in
\mathbb{N}$, define distributional function
$\phi^{n}_{xy}(F,\cdot):\mathbb{R}^+\rightarrow [0,1]$
by
\begin{equation*}
\phi^{n}_{xy}(F,t)=\frac{1}{n}Card\{0\leq i\leq n-1:
d_H(F^{i}(x),F^{i}(y))<t\}.
\end{equation*}

Let
\begin{align*}
\phi_{xy}(F,t)=\liminf\limits_{n\rightarrow\infty}\phi^{n}_{xy}(F,t), \\
\phi_{xy}^{*}(F,t)=\limsup\limits_{n\rightarrow\infty}\phi^{n}_{xy}(F,t).
\end{align*}
Then we call $\phi_{xy}(F,t)$ and $\phi_{xy}^{*}(F,t)$ are lower distributional function and upper distributional function generated by $F,x,y$ respectively.

\begin{definition}
Let $(X,d)$ be a compact metric space and $G$ be a semigroup generated by a finite set $F=\{f_1,f_2,\ldots,f_p\}$.
A pair $\{x,y\}\subset X$ is called Hausdorff metric distributionally chaotic of type
$k\in\{1,2,3\}$ (briefly, HDC1, HDC2 and DC3, respectively), if it
satisfies condition $(k)$ as follows

(1)$\phi_{xy}^{*}\equiv 1$ and $\exists \ \tau_0>0, \
\phi_{xy}(F,\tau_0)=0.$

(2)$ \phi_{xy}^{*}\equiv 1$ and $\phi_{xy}^{*}>\phi_{xy}.$

(3)$\phi_{xy}^{*}>\phi_{xy}.$

Furthermore, $F$ is called Hausdorff metric distributionally chaotic of type
$k\in\{1,2,3\}$, if there exists an uncountable subset $D\subseteq
X$ such that each pair of two distinct points is a Hausdorff metric distributionally
chaotic pair of type $k$. Moreover, $D$ is called a Hausdorff metric distributionally
scrambled set of type $k$.
\end{definition}

For a single continuous map, Blanchard et al. \cite{Bla} used
ergodic methods to prove that positive
entropy implies Li-Yorke chaos, later, Kerr and Li gave a combinatorial proof \cite{Kerr} to this result; T. Downarowicz \cite{Down}
proved that positive
entropy implies distributional chaos of type 2. Now we may ask "What is the relationship between positive Hausdorff metric entropy and Hausdorff metric chaos?"

\begin{problem}
Let $(X,d)$ be a compact metric space and $G$ be a semigroup generated by a finite set $F=\{f_1,f_2,\ldots,f_p\}$. Does
positive entropy imply Hausdorff metric Li-Yorke chaos for $F$?
\end{problem}

\begin{problem}
Let $(X,d)$ be a compact metric space and $G$ be a semigroup generated by a finite set $F=\{f_1,f_2,\ldots,f_p\}$. Does
positive entropy imply Hausdorff metric distributional chaos of type 2 for $F$?
\end{problem}
The answers seems to be positive. However, we are not able to simply use the methods appeared in \cite{Bla, Kerr, Down} to solve them.

As well-known, for a homeomorphism $f:X\rightarrow X$, the equality $h(f^{-1})=h(f)$ holds. Then how about the Hausdorff metric entropy? (This question for topological entropy of a semigroup defined by Bi\'{s} is also unsolved.)
Let $(X,d)$ be a compact metric space and $G$ be a semigroup generated by a finite set $F=\{f_1,f_2,\ldots,f_p\}$. If $f_j$ is invertible for all $1\leq j\leq p$, let $F^{-1}=\{f_1^{-1},f_2^{-1},\ldots,f_p^{-1}\}$ and $G^{-1}$ be the semigroup generated by $F^{-1}$.
\begin{problem}
Let $(X,d)$ be a compact metric space and $G$ be a semigroup generated by a finite set $F=\{f_1,f_2,\ldots,f_p\}$, where $f_j$ is invertible for all $1\leq j\leq p$. Does $h_H(G^{-1}, F^{-1})=h_H(G,F)$
hold?
\end{problem}

\begin{problem}
What are the relations between the Hausdorff metric entropy and other entropy-like invariants?
\end{problem}
From Proposition \ref{HB} and Remark \ref{rel}, we have known the relations between $h_H(G,F)$, $h(G,F)$ and $\max \{h(f_i); f_i\in F\}$.

\end{document}